\documentclass[12pt,centertags]{amsart}

\usepackage{graphicx,picture}
\usepackage{amsmath,amssymb,hyperref}

\usepackage[top=3.2cm, bottom=3.2cm, left=2.8cm,
right=2.8cm]{geometry}

\newtheorem{theo}{Theorem}[section]
\newtheorem{lemma}{Lemma}[section]
\newtheorem{prop}{Proposition}[section]

\theoremstyle{definition}
\newtheorem{definiz}{Definition}[section]
\newtheorem{rem}{Remark}[section]
\newtheorem{ex}{Example}[section]
\numberwithin{equation}{section}

\newcommand{\R}{\mathbb R}
\newcommand{\de}{\partial}

\DeclareMathOperator{\divergenza}{div}
\DeclareMathOperator{\epi}{epi}

\begin{document}
\title[The relative isoperimetric inequality: the anisotropic
case]{relative isoperimetric inequality in the plane: the
  anisotropic case}
\author[F. Della Pietra, N. Gavitone]{Francesco Della Pietra and
  Nunzia Gavitone}
\address[Francesco Della Pietra]{
Universit\`a degli studi del Molise \\
Dipartimento S.A.V.A. \\
Facolt\`a di Ingegneria \\
Via Duca degli Abruzzi \\
86039 Termoli (CB), Italia.
}
\email{francesco.dellapietra@unimol.it}

\address[Nunzia Gavitone]{
Universit\`a degli studi di Napoli ``Federico
II''\\ 
Dipartimento di Ma\-te\-ma\-ti\-ca e Applicazioni
``R. Cac\-ciop\-po\-li''\\ 80126 Napoli, Italia.
}
\email{nunzia.gavitone@unina.it}

\maketitle
\begin{abstract}
In this paper we prove a relative isoperimetric inequality in  the
plane, when the perimeter is defined with respect to a convex, positively
 homogeneous function of degree one $H\colon\R^2\rightarrow [0,+\infty[$.
Under suitable assumptions on $\Omega$ and $H$, we also characterize
the minimizers.
\end{abstract}

\section{Introduction}
Let $\Omega$ be an open bounded connected set of $\R^2$, with
Lipschitz boundary. The classical relative isoperimetric inequality
states that
\begin{equation}\label{intro:reliso}
P^2(E;\Omega) \ge C \min\{|E|, |\Omega \setminus E|\},
\end{equation}
for any measurable subset $E$ of $\Omega$ (see, for example,
\cite{ff},\cite{maz},\cite{bz}). Here $|E|$ is the Lebesgue measure of
$E$, and $P(E;\Omega)$ is the usual perimeter in $\Omega$.
Being $P(E;\Omega)=P(\Omega\setminus E;\Omega)$, the inequality
\eqref{intro:reliso} can be written as
\begin{equation}\label{intro:reliso2}
P^2(E;\Omega) \ge C |E|,
\end{equation}
for any $E\subset\Omega$ such that $|E|\le |\Omega|/2$.

Natural questions related to the inequality \eqref{intro:reliso2} are
the following: finding the optimal constant
\begin{equation}\label{intro:ratio}
C(\Omega) = \inf \left\{ \frac{P^2(E;\Omega)}{|E|} \colon 0<|E|\le
\frac{|\Omega|}{2},\;E\subseteq \Omega \right\},
\end{equation}
proving that it is attained, and characterizing the minimizers.

First results in this direction can be found in \cite{bz} or
\cite{maz}, where it is proved that $C(\Omega) =\frac 8 \pi$ when
$\Omega$ is the unit disk in $\R^2$, and it is attained at a
semicircle. More generally, in \cite{cia} the author proves that for
an open convex set $\Omega$ of the plane, $C(\Omega)$ is actually a
minimum. Moreover, there exists a convex minimizer of
\eqref{intro:ratio} whose measure equals $\frac{|\Omega|}{2}$, and any
minimizer $E$ has the following 
properties:
\begin{itemize}
\item[(a)] $\de E \cap \Omega$ is either a circular arc or a straight
  segment. Moreover, neither $E$ nor $\Omega\setminus E$ is a circle.
\item[(b)] Let $T$ be one of the terminal points of $\de E \cap
  \Omega$. Then $T$ is a regular point of $\de \Omega$ and $\de E \cap
  \Omega$ is orthogonal to $\de \Omega$. As a consequence, either $E$
  or $\Omega\setminus E$ is convex.
\item[(c)] If $|E|<\frac{|\Omega|}{2}$, then $E$ is a circular sector
  having sides on $\de \Omega$. In such a case, there exists another
  minimizer $F$ which is a sector with sides on $\de \Omega$, having
  the same vertex as $E$, such that $|F|=\frac{|\Omega|}{2}$.
\end{itemize}
Furthermore, in \cite{cia} $C(\Omega)$ is explicitly computed under
the additional assumption that $\Omega$ is symmetric about a point
and also in special cases of convex domains. If $r(\Omega)$ is the
inradius of $\Omega$, then
\[
C(\Omega)=\frac{8 r^2(\Omega)}{|\Omega|}.
\]

We refer the reader to \cite{efknt} for some extremal
problems involving $C(\Omega)$. 

The purpose of the present paper is to find analogous results when
the Euclidean perimeter is replaced by an ``anisotropic'' perimeter.
More precisely, if $H$ is an arbitrary norm on $\R^2$, the perimeter
with respect to $H$ for a set $E\subseteq \R^2$ with sufficiently
smooth boundary is given by
\[
P_H(E;\Omega)=\int_{\de E\cap \Omega} H(\nu_E) \,d \mathcal H^1,
\]
where $\mathcal H^1$ is the 1-dimensional Hausdorff measure and
$\nu_E$ is the unit outer normal to $E$ (see Section 2 for the precise
definition). 

We recall that in this setting it is well-known that the following
isoperimetric inequality holds for any $E\subseteq \R^2$
\begin{equation}\label{eq:abs}
P_H^2(E;\R^2)\ge 4 |W| |E|,
\end{equation}
where $W=\{(x,y)\colon H^o(x,y)<1\}$ and $H^o$ is polar to $H$ (see
\cite{bus},\cite{dp},\cite{fm},\cite{aflt},\cite{str}).
Moreover, the equality in \eqref{eq:abs} holds if and only if $E$ is
homothetic to $W$. We refer to $W$ as the Wulff shape.

Our results can be summarized as follows. Under suitable assumptions
on $H$, we first show that an anisotropic relative isoperimetric
inequality holds. That is: when $\Omega$ is an open, bounded
connected set of $\R^2$, with Lipschitz boundary, then there exists
$C_H(\Omega)>0$ such that
\begin{equation}\label{aniso}
C_H(\Omega)= \inf \left\{ \frac{P_H^2(E;\Omega)}{|E|} \colon 0<|E|\le
\frac{|\Omega|}{2},\;E\subseteq \Omega \right\}.
\end{equation}
Then we prove that, for a convex set $\Omega$, $C_H(\Omega)$ is
actually a minimum, there exists a convex minimizer of \eqref{aniso}
 whose measure equals $\frac{|\Omega|}{2}$, and any minimizer $E$ has
 the following properties:
\begin{itemize}
\item[($\alpha$)] $\de E \cap \Omega$ is either homothetic to a Wulff
  arc (that is an arc of $\de W$) or a straight segment. Moreover,
  neither $E$ nor $\Omega\setminus E$ is homothetic to a Wulff shape.
\item[($\beta$)]  Let $T$ be one of the terminal points of $\de E \cap
  \Omega$. Then $T$ is a regular point of $\de \Omega$ and $\de E \cap
  \Omega$ verifies the following contact angle condition with $\de
  \Omega$:
  \[
  \langle \nabla H(\nu_E), \nu_\Omega \rangle = 0,
  \]
where $\nu_\Omega$ and $\nu_E$ are the usual unit outer normal vectors to
$\de \Omega$ and $\de E$ at $T$ respectively.
\item[($\gamma$)] If $|E|<\frac{|\Omega|}{2}$, then $E$ is homothetic
  to a Wulff sector (see section 2 for the precise definition) having
  sides on $\de \Omega$. In such a 
  case, there exists another minimizer $F$ which is a sector with
  sides on $\de \Omega$, having the same vertex as $E$, such that
  $|F|=\frac{|\Omega|}{2}$. 
\end{itemize}
Furthermore, we explicitly compute $C_H(\Omega)$ under the additional
assumption that $\Omega$ is symmetric about a point. Indeed,
\[
C_H(\Omega)=\frac{8 r_H^2(\Omega)}{|\Omega|},
\]
where $r_H(\Omega)$ is defined in Theorem~\ref{thm:cost}. For example,
if $\Omega$ is obtained by a rotation of $\frac \pi 2$ of a level set
of $H$, that is $\Omega=\{(x,y)\colon H(-y,x)<r\}$,
then
\[
C_H(\Omega)=\frac{8 r^2}{|\Omega|}=\frac {8}{\kappa_H},
\]
where $\kappa_H=|\{(x,y)\colon H(x,y)<1\}|$. We recover immediately
the classical result $C_H=8/\pi$ when $H$ is the Euclidean norm.

The paper is organized as follows. In Section 2 we give the precise
definitions of anisotropic perimeter and some basic properties. In
Section 3 we prove the main result. A fundamental argument
is to study problem \eqref{aniso} by considering the area $|E|$ fixed.
Finally, we give some examples.

\section{Notation and preliminaries}
Let $H:\R^2\rightarrow [0,+\infty[$ be a $C^2(\mathbb
R^2\setminus\{0\})$ function such that $H^2(\xi)$ is strictly convex
and
\begin{equation}\label{eq:omo}
  H(t\xi)= |t| H(\xi), \quad \forall \xi \in \R^2,\; \forall t \in \R.
\end{equation}
Moreover, suppose that there exist two positive constants $\alpha
\le \beta$ such that
\begin{equation}\label{eq:lin}
  \alpha|\xi| \le H(\xi) \le \beta|\xi|,\quad \forall \xi\in \R^2.
\end{equation}

We define the polar function $H^o\colon
\R^2\rightarrow [0,+\infty[$ of $H$ as
\[
H^o(v)=\sup_{\xi \ne 0} \frac{\langle \xi, v\rangle}{H(\xi)}
\]
where $\langle\cdot,\cdot\rangle$ is the usual scalar product of
$\R^2$. It is easy to verify that also $H^o$ is a convex function
which satisfies properties \eqref{eq:omo} and
\eqref{eq:lin}. Furthermore,
\[
H(v)=\sup_{\xi \ne 0} \frac{\langle \xi, v\rangle}{H^o(\xi)}.
\]
The set 
\[
W = \{  \xi \in \R^2 \colon H^o(\xi)< 1\}
\]
is the so-called Wulff shape centered at the origin.

We will call Wulff sector with vertex at the origin the set 
$A\cap W$, where $A$ is an open cone with vertex at $(0,0)$.

The following properties of $H$ and $H^o$ hold true
(see for example \cite{bp}):
\begin{gather}
 H(\nabla H^o(\xi))=H^o(\nabla H(\xi))=1,\quad \forall \xi \in
\R^2\setminus \{0\},\\
H^o(\xi) \nabla H(\nabla H^o(\xi) ) = H(\xi) \nabla
H^o(\nabla H(\xi) ) = \xi,\quad \forall \xi \in
\R^2\setminus \{0\}. \label{eq:HH0}
\end{gather}
\begin{definiz}[Anisotropic relative perimeter] Let $\Omega$ be an
  open bounded set of $\mathbb R^2$. In \cite{ab}, the perimeter of
  $F\subset \R^2$ in $\Omega$ with respect to $H$ is defined
  as the quantity
\[
P_H(F;\Omega) = \sup\left\{ \int_F \divergenza \sigma dx\colon
  \sigma \in C_0^1(\Omega;\R^2),\; H^o(\sigma)\le 1 \right\}.
\]
\end{definiz}
The equality
\[
P_H(F;\Omega)= \int_{\Omega\cap \partial^*F} H(\nu_F) d\mathcal H^1
\]
holds, where $\partial^*F$ is the reduced boundary of $F$ and $\nu_F$ is
the unit outer normal to $F$ (see \cite{ab}).

The anisotropic perimeter of a set $F$ is
finite if and only if the usual Euclidean perimeter $P(F;\Omega)$
\[
P(F;\Omega)=  \sup\left\{ \int_F \divergenza \sigma dx\colon
  \sigma \in C_0^1(\Omega;\R^N),\; |\sigma|\le 1 \right\}.
\]
 is finite. Indeed, by properties \eqref{eq:omo} and \eqref{eq:lin} we
have that
\[
\frac{1}{\beta} |\xi| \le H^o(\xi) \le \frac{1}{\alpha} |\xi|,
\]
and then
\begin{equation}\label{eq:per}
\alpha P(E;\Omega) \le P_H(E;\Omega) \le \beta P(E;\Omega).
\end{equation}

\begin{rem}\label{rem:corda}
We observe that when $\de E\cap \Omega$ is the image of a smooth
curve $\gamma(t)=(x(t),y(t))$, $t\in[a,b]$, then $P_H(E;\Omega)$
coincides with the value
  \begin{equation}\label{length}
  \mathcal L_H(\gamma)=\int_a^b H(-y'(t),x'(t))\,dt.
  \end{equation}

By regularity of $H$, the curve joining two points $P_0$ and $P_1$
which minimizes $\mathcal L_H$ is the straight segment ${P_0
  P_1}$. This can be shown by classical argument of Calculus of
Variations. We consider, for sake of simplicity, the curves
$\gamma(t)=(t,u(t))$. Denoting by $\mathcal L_H(u)=\mathcal
L_H(\gamma)$, the minimum of the problem 
\[
\left\{
  \begin{array}{l}
    \min \mathcal L_H(u), \\
    u(a)=u_a, \; u(b)=u_b,
  \end{array}
\right.
\]
is the solution to
\[
\left\{
  \begin{array}{l}
\frac{d}{dt} H_x(-u'(t),1)=0, \\
u(a)=u_a, \; u(b)=u_b.
  \end{array}
\right.
\]
Such solution is the linear function passing through $P_0=(a,u_a)$ and
$P_1=(b,u_b)$.
\end{rem}

\begin{definiz}[Anisotropic curvature (\cite{atw},\cite{bp})]
  Let
  $F\subset \R^2$ be a bounded open set with smooth boundary,
  $\nu_F(x,y)$ the unit outer normal at $(x,y)\in \partial F$, in the
  usual Euclidean sense. Let $u$ be a $C^2$ function such that
  $F=\{u>0\}$, $\de F = \{ u=0 \}$ and $\nabla u \neq (0,0)$ on $\de
  F$. Hence, $\nu_F=-\frac{\nabla u}{|\nabla u|}$ on $\de F$.
  The anisotropic outer normal $n$ is defined as
  \[
  n_F(x,y)= \nabla H(\nu_F(x,y))=\nabla H\left(-\frac{\nabla u}{|\nabla
      u|}\right),\quad (x,y)\in \partial F,
  \]
  and, by the properties of $H$,
  \[
  H^o(n_F)=1.
  \]
  The anisotropic curvature $k_H$ of $\partial F$ is
  \[
  k_H(x,y)= \divergenza n_F(x,y)=\divergenza\left[ \nabla
    H\left(-\frac{\nabla u}{|\nabla u|}\right) \right], \quad (x,y)\in
  \de F.
  \]

  Let $(x_0,y_0)\in \de F$. Without loss of generality, we can locally
  describe $\de F$ with a $C^2$ function
  $v\colon]x_0-\delta,x_0+\delta[\rightarrow \R$, that is $F$ is the
  epigraph of $v$ near $(x_0,y_0)=(x_0,v(x_0))$. By properties of 
  $H$, the anisotropic curvature $k_H(x_0,y_0)$ of $\partial F$ at
  $(x_0,y_0)$ can be written as 
  \[
  k_H(x_0,y_0)= - \left. \dfrac{d}{d t}  H_x(-v'(t),1) \right |_{t=x_0}.
  \]
\end{definiz}

\begin{rem}
We stress that if $F$ is homothetic to the Wulff shape $W$ and
centered at $(x_0,y_0)$, the anisotropic outer normal at $(x,y)\in \de F$
has the direction of $(x-x_0,y-y_0)$. Indeed, being $F=\{(x,y)\colon
H^o(x-x_0,y-y_0)=\lambda\}$, for some positive $\lambda$, by property
\eqref{eq:HH0} it follows that
\[
n_F(x,y)= \nabla H\big(\nabla H^o(x-x_0,y-y_0) \big) =
\frac{1}{\lambda}(x-x_0,y-y_0).
\]
See Figure \ref{fig:normal} for an example.
\end{rem}
\begin{center}
\begin{figure}[h]
\includegraphics{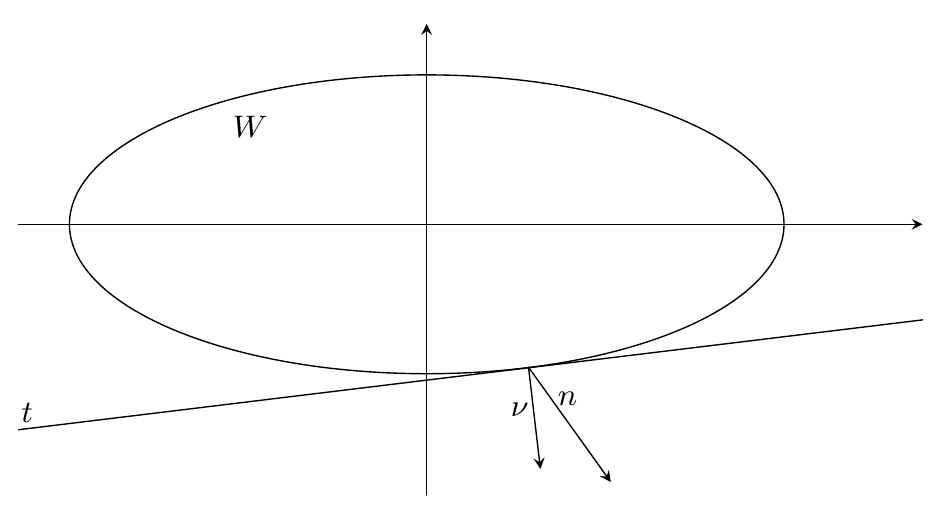}
\caption{ Here $H(x,y)=(
  {x^2}/{a^2}+ {y^2}/{b^2})^{1/2}$ and  $H^o(x,y)=(
  {a^2}{x^2}+ {b^2}{y^2})^{1/2}$. When $a\ne b$, the usual and the
  anisotropic outer normal are, in general, different.}\label{fig:normal}
\end{figure}
\end{center}

\begin{rem}\label{rem:constcurv}
Let be $F=\frac 1 \lambda W$, with $\lambda >0$. It is not difficult to
show (see, for instance, \cite{bnp}, \cite{bp}) that the anisotropic
curvature at $(x,y)\in \partial F$ is
\[
k_H(x,y)= \lambda.
\]
\end{rem}
\section{An anisotropic relative isoperimetric inequality}

\begin{theo}
Let $\Omega$ be an open bounded connected set of $\R^2$, with Lipschitz
boundary. Then an anisotropic relative isoperimetric inequality
holds. Namely, there exists a constant $C>0$ such that
\begin{equation}\label{eq:reliso}
P_H^2(E;\Omega) \ge C \min\{|E|, |\Omega \setminus E| \},
\end{equation}
for every measurable set $E\subseteq \Omega$.
\end{theo}
\begin{proof} The hypotheses made on $\Omega$ guarantee that a
  relative isoperimetric inequality holds when we consider the
  usual perimeter $P(E;\Omega)$ (see \cite{ff},\cite{maz},\cite{cia}).
 Hence the inequality \eqref{eq:reliso} follows immediately from
 property \eqref{eq:per}.
\end{proof}
Our aim is to study, for $\Omega$ bounded and convex, the best
constant in the inequality \eqref{eq:reliso}, that is to find the
infimum
\begin{equation}\label{eq:ratio}
C_H=\inf \left\{ Q(E)\colon 0<|E|\le
\frac{|\Omega|}{2},\;E\subseteq \Omega \right\},
\end{equation}
where
\[
Q(E)=\frac{P_H^2(E;\Omega)}{|E|},
\]
to prove that $C_H$ is actually a minimum, and to characterize the
minimizers. Furthermore, we will find the explicit value of $C_H$ in
some special case.

If $E$ is a minimizer of \eqref{eq:ratio}, then $E$ solves
also the following problem under volume constraint:
\[
  \min \{P_H(F;\Omega),\, F\subset \Omega\text{ and }|F|= |E| \}.
\]
The following result characterizes the minimizers of the above
problem.
\begin{theo}\label{th:Wulff-arc}
Let $\Omega$ be an open bounded connected set of $\R^2$, with
Lipschitz boundary. Then there exists a minimizer $E$ of the problem
\begin{equation}\label{eq:pb}
  \min \{P_H(F;\Omega),\, F\subset \Omega\text{ and }|F|= k \},
\end{equation}
with $0<k\le |\Omega|/2$ fixed. Moreover, $\partial E\cap \Omega$ is
either homothetic to an arc of $\partial W$, or a straight
segment. Hence a minimizer of \eqref{eq:ratio}, if   exists, has the
same characterization.
\end{theo}

\begin{proof}
The existence of a minimizer of \eqref{eq:pb} follows by the
lower semicontinuity of $P_H$ (see \cite{ab}) using standard methods
of Calculus of Variations. 

To prove the result, we proceed by steps.

\noindent {\bf Step 1.}
First, we show that a minimizer $E$ is locally homothetic to an arc
of $\partial W$, or a straight segment.

Fixed $(x_0,y_0)\in \de E\cap \Omega$, we can locally describe $\de
E\cap \Omega$ with a $C^2$ function $u$ (see
\cite{atw},\cite{b},\cite{anp},\cite{np}). That is, without loss of 
generality, there exists a rectangle $R=]x_0-\delta, x_0+\delta[\times
I $ where $E\cap R$ is the epigraph of $u: ]x_0-\delta,
x_0+\delta[\to I$. Moreover, there exists $\lambda$ such that
$u$ is the minimum of the functional
\[
J(v)= \int_{x_0-\delta}^{x_0+\delta} H(-v'(t) , 1) dt
+ \lambda \int_{x_0-\delta}^{x_0+\delta} v(t) dt,
\]
with boundary conditions $v(x_0+\delta)=u(x_0+\delta)$ and
$v(x_0-\delta)=u(x_0-\delta)$. The corresponding Euler equation
associated to $J$ is
\begin{equation}\label{eq:euler}
\left\{
\begin{array}{l}
-\dfrac{d}{d t}  H_x(-v'(t),1)  = \lambda,
\quad t\in ]x_0-\delta, x_0+\delta[,\\
v(x_0\pm\delta)=u(x_0\pm\delta).
\end{array}
\right.
\end{equation}
If $\lambda=0$, there exists a linear function $u_0$ which solves
\eqref{eq:euler}. If $\lambda\neq 0$, by Remark \ref{rem:constcurv},
the function $u_\lambda(t)$, which describes $\frac 1 \lambda \de W$
(up to translation) near $x_0$, is a solution of
\eqref{eq:euler}. On the other hand,  for any $\lambda\in \R$, the
regularity on $H$ guarantees that the functional $J$ is strictly
convex. Hence, $u_\lambda=u$ is the unique
solution of \eqref{eq:euler} (see also \cite{bnp}, \cite{np}).

\noindent {\bf Step 2.}
Now we show that the minimizer has the same anisotropic curvature at
any point.

Let us take $(x_1,y_1)$ and $(x_2,y_2)$ in $\de E \cap \Omega$.
As in the step 1, let us consider $u_1\colon B_1=]x_1-\delta_1, x_1 +
\delta_1[\to I_1$ and $u_2\colon B_2=]x_2-\delta_2, x_2 + \delta_2[\to
I_2$ two functions which locally describe $\de E \cap
\Omega$. Moreover, there exist $\lambda_1$ and $\lambda_2$ such that
$u_i$, for $i=1,2$,  minimizes the functional
\[
J_i(v)= \int_{B_i} H(-v'(t) , 1) dt
+ \lambda_i \int_{B_i} v(t) dt, \quad i=1,2,
\]
with boundary conditions $v(x_i\pm\delta_i)=u_i(x_i\pm\delta_i)$.
We claim that $\lambda_1=\lambda_2$. This can be shown by arguing as
in \cite{gmt}, Theorem 2. We briefly describe the idea, and we refer
to the quoted paper for the precise details.

We assume that $0\le \lambda_1 < \lambda_2$. A similar argument can be
repeated in the other cases.

For every $\lambda\in ]\lambda_1,\lambda_2[$ there exists a function
$u_{\rho,i}$ which is the unique minimizer to
\[
\int_{B_{\rho,i}} H(-v'(t) , 1) dt
+ \lambda \int_{B_{\rho,i}} v(t) dt, \quad i=1,2,
\]
where $0<\rho < \min_i \delta_i$ and $B_{\rho,i}=]x_i-\rho,x_i+\rho[$,
with boundary conditions $v(x_i\pm\rho)=u_i(x_i\pm\rho)$.

By convexity of $H$, a comparison argument shows that $u_{\rho,1}\le
u_1$ in $B_{\rho,1}$, and  $u_{\rho,2}\ge u_2$ in $B_{\rho,2}$.
Defining
\[
V_{\rho,i}= \int_{B_{\rho,i}} \lvert u_i -u_{\rho,i} \rvert dt,
\]
it is possible to prove that there exist two suitable positive numbers
$r_1$ and $r_2$ such that
\begin{equation}\label{eq:aree}
V_{r_1,1}=V_{r_2,2}.
\end{equation}
This implies that, defining the set $E^*$ as
\[
E^*=\left[ E\cup ( \epi u_{r_1,1} \cap C_1 ) \right] \setminus 
\left[ C_2\cap(E\setminus\epi u_{r_2,2}) \right],
\]
where $C_i= B_i\times I_i$, we have that $|E^*|=|E|$.

Finally, we get that $E \Delta E^* \Subset \Omega $ and
\begin{multline*}
P_H(E; \Omega) - P_H(E^*;\Omega) = \\
= \int_{B_{r_1,1}} H(-u_1',1) dt + \int_{B_{r_2,2}} H(-u_2',1) dt + \\
- \int_{B_{r_1,1}} H(-u_{r_1,1}',1) dt - \int_{B_{r_2,2}}
H(-u_{r_2,2}',1) dt + \\ + \lambda \int_{B_{r_1,1}} (u_1- u_{r_1,1})
dt +   \lambda \int_{B_{r_2,2}} (u_2- u_{r_2,2}) dt,
\end{multline*}
where last line in the above equality vanishes, by \eqref{eq:aree}.

By minimality of $u_{r_1,1}$ and $u_{r_2,2}$, $P_H(E; \Omega) >
P_H(E^*;\Omega)$, and this contradicts the minimality of $E$. Hence,
$\lambda_1=\lambda_2$.

\noindent {\bf Step 3.} We point out that the claim of Step 2 assure
that $\de E \cap \Omega$ consists of Wulff arcs, all with the same
curvature, or straight segments. To conclude the proof of the Theorem,
we have to prove that $E$ and $\partial E\cap \Omega$ are connected.
This can be shown by repeating line by line the proof of Theorem 2
in \cite{cia}.
\end{proof}

The following property of the minimizers is a direct consequence of
Remark \ref{rem:corda}.
\begin{prop}\label{prop:conv}
  Let $\Omega$ be an open bounded connected set of $\R^2$, with
  Lipschitz boundary. Suppose that $E$ is a minimizer of
  \eqref{eq:ratio}. If $|E|<|\Omega|/2$ and $\de E \cap \Omega$ is not
  a straight segment, $\de E \cap \Omega$ is concave towards $E$.
\end{prop}
\begin{proof}
  If $|E|<|\Omega|/2$ and $\de E \cap \Omega$ is strictly concave
  towards $\Omega\setminus E$, we can consider a new set $E^*$ by
  adding to $E$ the region of $\Omega$ between  $\de E\cap
  \Omega$ and a straight segment joining two suitable points of $\de
  E\cap \Omega$. Choosing the two points sufficiently near, we get
  that $|E|\le |E^*|\le |\Omega|/2$ and, by Remark \ref{rem:corda},
  $P_H(E^*;\Omega)< P_H(E;\Omega)$. This contradicts the minimality of
  $E$.
\end{proof}
\begin{theo}\label{thm:angle}
Let $\Omega$ be an open bounded convex set of $\R^2$. Suppose that $E$
is a minimizer of \eqref{eq:ratio}, and let $T$ be a terminal
point of $\partial E$. Then $\de \Omega$ at $T$ is $C^1$, and
\begin{equation}\label{eq:angle}
\langle n_E , \nu_\Omega \rangle = 0
\end{equation}
where $n_E$ is the anisotropic outer normal to $\partial E$ and
$\nu_\Omega$ is the usual unit outer normal to $\partial \Omega$ 
at $T$.
\end{theo}

\begin{rem}\label{rem:angle}
The angle condition is justified by the following natural geometric
argument.

Let $ s \colon \alpha_s x + \beta_s y + q_s = 0 $ be a straight line,
$P_0=(x_0,y_0)\in \R^2\setminus s$. By an immediate calculation, the
straight segment which minimizes $L_H$ between $P_0$ and $s$
is parallel to the straight line $r \colon \alpha_r x + \beta_r y=0$
which has to satisfy the following orthogonality condition:
\begin{equation}\label{eq:angletg}
  \langle \nabla H(\beta_r,\alpha_r) , (\beta_s, \alpha_s) \rangle = 0.
\end{equation}

Using the notation of Theorem \ref{thm:angle}, if we consider as $r$
the tangent line to $\de\Omega$ at a terminal point $T$ of $\de E
\cap \Omega$, and as $s$ the tangent straight line to $\de E$ at $T$,
then \eqref{eq:angle} and \eqref{eq:angletg} coincide.
\end{rem}

\begin{figure}[h]
\begin{center}
 \includegraphics{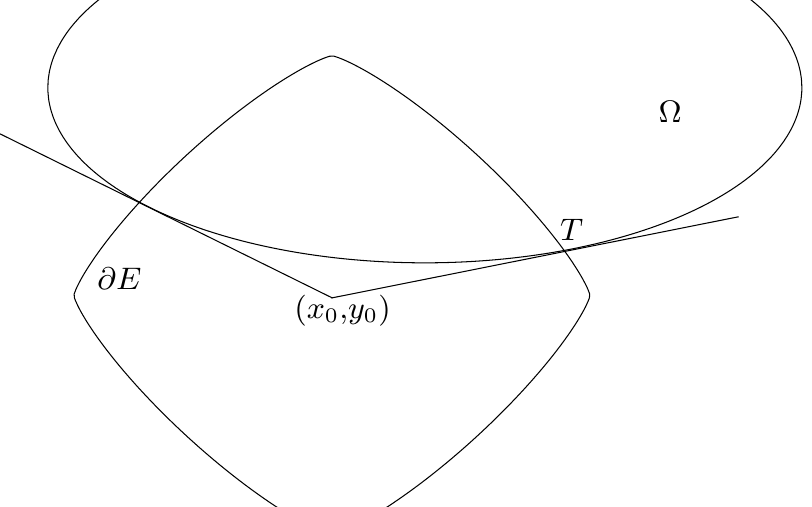}
\end{center}
\caption{Contact angle condition, with $H(x,y) \!\!=\!(x^4+y^4)^{\frac
   1 4}$ and $E$ is homothetic to the Wulff shape $W$ and centered at
  $(x_0,y_0)$. The tangent lines to $\de \Omega$ at the contact points
  have the same direction of the anisotropic normal to $\de E$ at the
  same points.
}
\end{figure}

\begin{proof}[Proof of Theorem \ref{thm:angle}]

We first assume that $\de \Omega$ is $C^1$ at $T$.

Let us suppose, by contradiction, that \eqref{eq:angle} is not
verified. The idea is to construct a new set $E^*$ such that
$Q(E^*)<Q(E)$. This will contradict the
minimality of $E$. To do that, we need to distinguish three cases.

First of all, we denote with $s$ the tangent line to $\de \Omega$ at
$T$, and with $t$ and $r$ two half line with vertex at $T$ and towards
$\Omega$ such that $t$ is tangent to $\de E$ at $T$ and $r$ satisfies
the angle condition \eqref{eq:angletg} with respect to $s$. 

\noindent {\bf Case 1.} We first assume that $|E|<|\Omega|/2$ and
the angle between $s$ and $t$ towards $E$ is greater than the one
between $s$ and $r$ towards $E$. We construct the straight segment ${Q
  Q_0}$ parallel to $r$ joining a suitable point $Q \in \de E \cap
\Omega$ and $Q_0\in s$. Being $\Omega$ convex, we can consider the
point $\bar Q = {QQ_0}\cap \de \Omega$. Denoted by $D$ the closed
region delimited by ${Q\bar Q}$, the arc of $\de E$ joining $Q$ and
$T$ and the arc of $\de\Omega$ between $T$ and $\bar Q$, let be
$E^*=E\cup D$ (see figure \ref{fig:3}).
\begin{figure}[h]
\begin{center}
\begin{picture}(\textheight,5.5cm) (-1.5cm,0cm)
  \includegraphics[scale=.6]{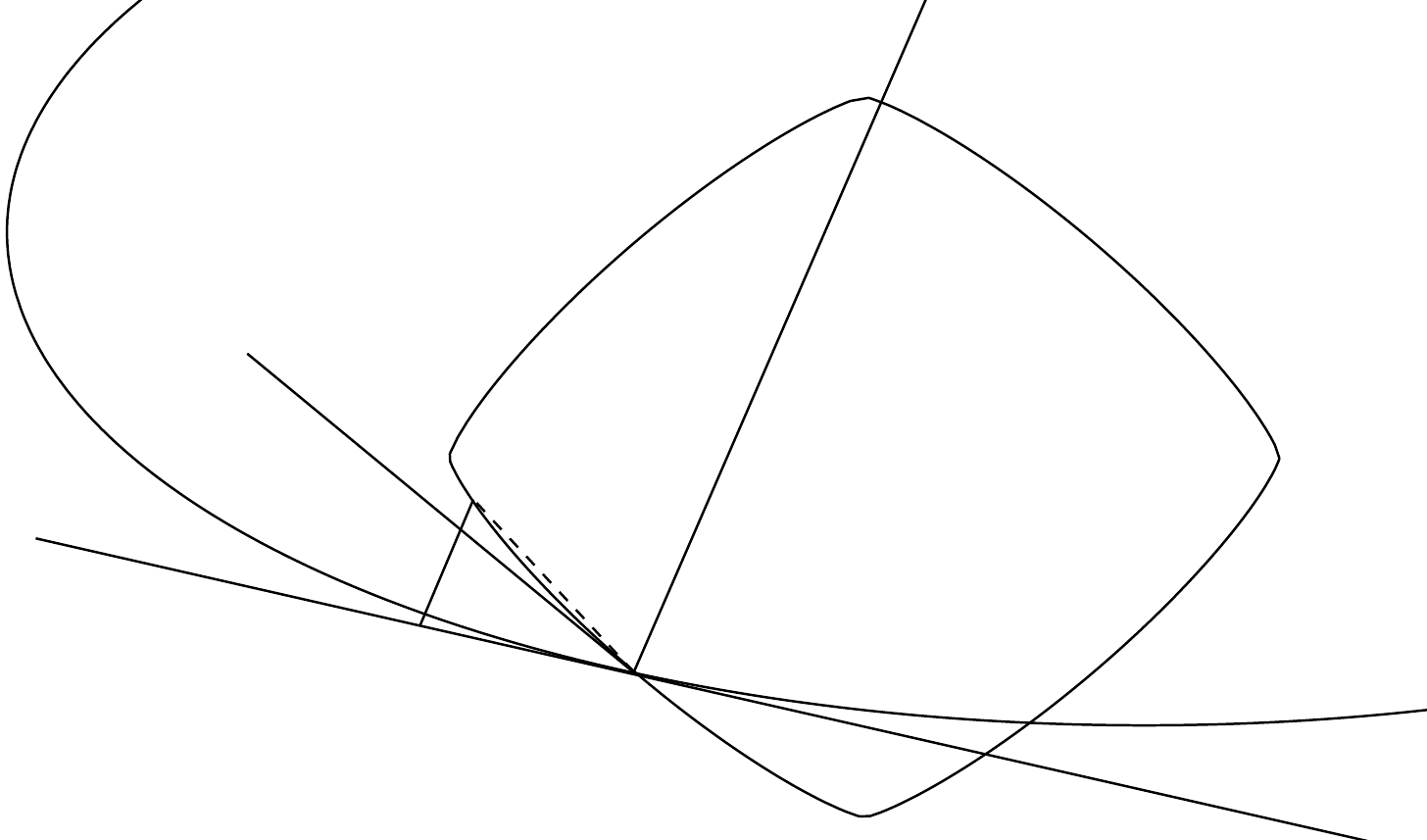}
  \put(-8cm,4.5cm){\scriptsize $\Omega$}
  \put(-2cm,2.6cm){\scriptsize $E$}
  \put(-5.1cm,0.77cm){\scriptsize $T$}
  \put(-.55cm,0.2cm){\scriptsize $s$}
  \put(-3.1cm,4.9cm){\scriptsize $r$}
  \put(-7.2cm,3cm){\scriptsize $t$}
  \put(-5.92cm,2.2cm){\scriptsize $Q$}
  \put(-6.4cm,1.09cm){\scriptsize $Q_0$}
  \put(-6.4cm,1.55cm){\scriptsize $\bar Q$}
  \put(-5.8cm,1.35cm){\scriptsize $D$}
\end{picture}
\caption{Case 1, construction of $E^*$.}\label{fig:3}
\end{center}
\end{figure}
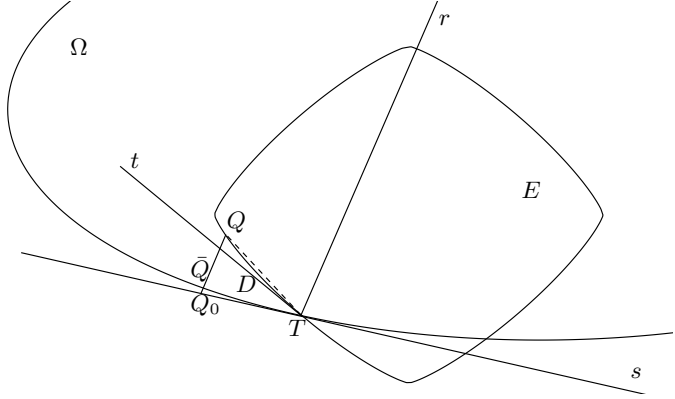

We choose $Q$ sufficiently near to $T$ such that
$|E^*|<|\Omega|/2$. Hence $E^*$ has larger area than $E$ and, by
Remark \ref{rem:angle} and Lemma \ref{rem:corda} also smaller
anisotropic perimeter.

\noindent {\bf Case 2.}
Now we still suppose that $|E|<|\Omega|/2$, and the angle between $s$
and $t$ towards $E$ is smaller that the one between $s$ and $r$. We
construct
the straight segment ${Q Q_0}$ parallel to $r$ joining a
suitable point $Q \in t$ and $Q_0\in s$, the point $\bar Q=QQ_0\cap
\de \Omega$ and the set $D$ as the intersection between the triangle
$QTQ_0$ and $E$ (see figure \ref{fig:4}). We define $E^*=E\setminus
D$.
\begin{figure}[h]
\includegraphics{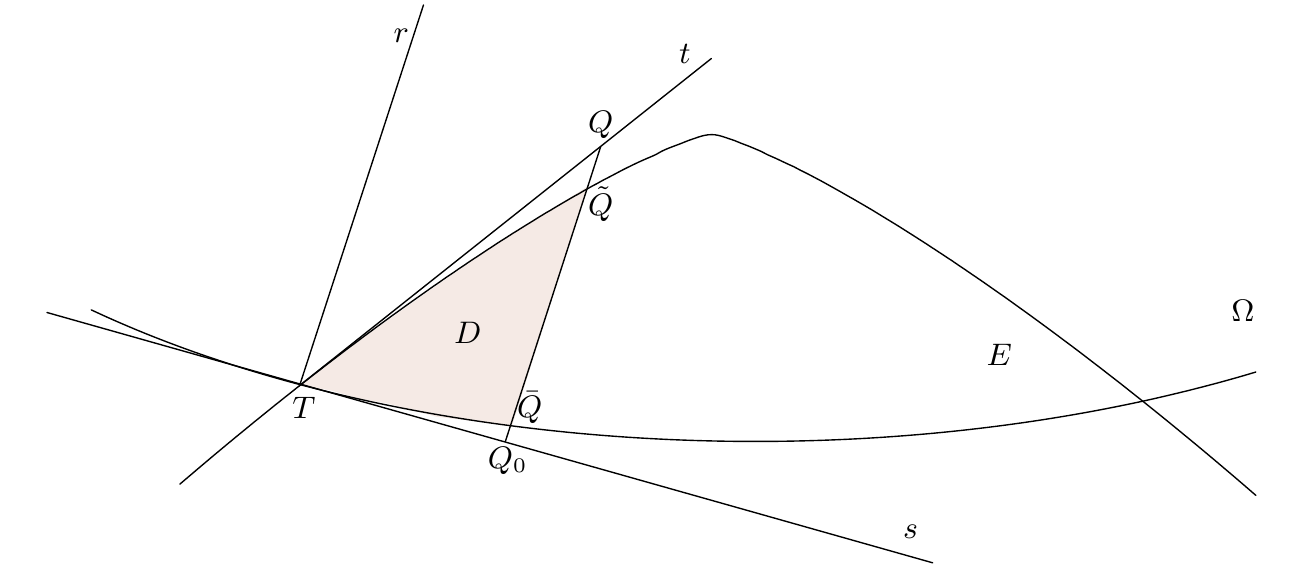}
\caption{Case 2, construction of $E^*$.}\label{fig:4}
\end{figure}

We show that, for a suitable choice of $Q$,
\begin{equation}\label{eq:contra}
  \frac{P^2_H(E;\Omega)}{|E|} > \frac{P_H^2(E^*;\Omega)}{|E^*|}.
\end{equation}
Differently from the case 1, inequality \eqref{eq:contra} is not
obvious because $E^*$ has both smaller perimeter and area. Hence, we
explicitly calculate the right-hand side in \eqref{eq:contra}.
Denoted by $ A=|E| $, $P=P_H(E;\Omega)$, $\delta A=|D|=|E|-|E^*|$,
$\delta P = P_H(E;\Omega) - P_H(E^*;\Omega)$, the inequality
\eqref{eq:contra} becomes
\begin{equation}\label{eq:contra2}
 \frac{P^2}{A} > \frac{(P -\delta P)^2}{A-\delta A}.
\end{equation}

Denoting by $l_{1,H}=\mathcal L_H(\gamma_{1}) $ and $l_{2,H}= \mathcal
L_H(\gamma_{2})$, where $\gamma_{1}$ and $\gamma_{2}$ are the
curves which represent $TQ$ and $QQ_0$ respectively,
it is easy to prove that
\[
l_{1,H}= l_1 \cdot H(-\beta,\alpha)=l_1 C_1,
\]
where $l_1$ and $(\alpha,\beta)$ are respectively the usual lenght
and the direction of $TQ$, and
\[
l_{2,H}= l_2\cdot H(-\beta_r,\alpha_r)=l_2C_2,
\]
where $l_2$ and $(\alpha_r,\beta_r)$ are respectively the usual lenght
and the direction of $QQ_0$. Observe that by construction,
$l_{1,H}>l_{2,H}$.

\begin{figure}[h]
\includegraphics{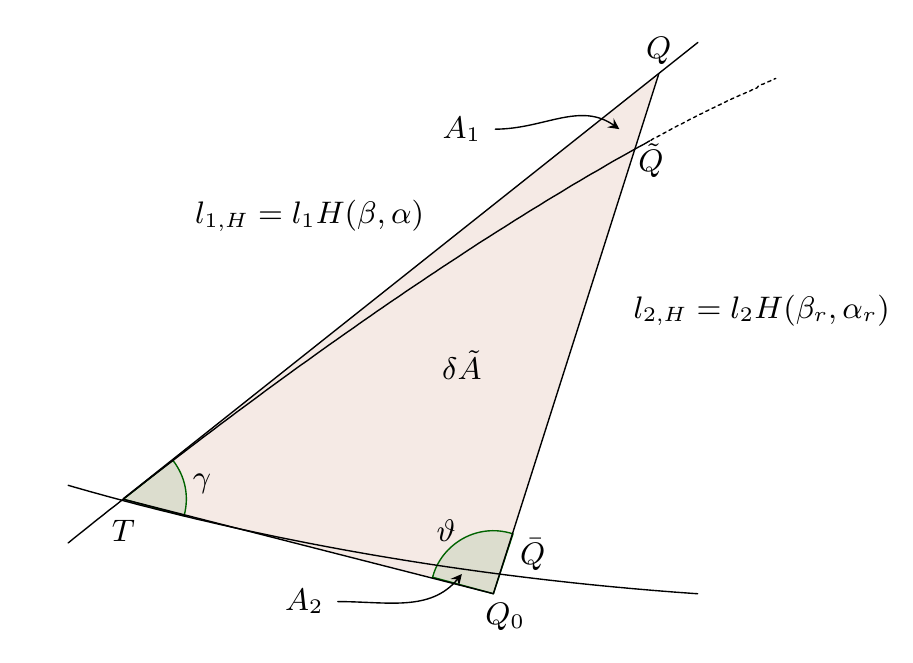}
\caption{Approximation in case 2.}\label{fig:approx}
\end{figure}

We first show \eqref{eq:contra2} replacing $\delta P$ with $\delta
\tilde P=l_{1,H}-l_{2,H}$ and $\delta A$ with $\delta \tilde A=\delta
A+A_1+A_2$, where $A_1$ and $A_2$ are the measures of the sets as in
figure \ref{fig:approx}.

By elementary properties of triangles,
\[
  \frac{(P-\delta \tilde P)^2}{A-\delta \tilde A}=
  \frac{(P-l_1C_1+l_2C_2)^2}{A-l_1l_2\sin(\gamma + \vartheta)} =
  \frac{\left(P-l_1\left(C_1-\frac{\sin \gamma}{\sin
          \theta}C_2\right)\right)^2}{A-l_1^2 \frac{\sin \gamma}{\sin
      \theta} \sin(\gamma+\vartheta) }=f(l_1)
\]
The function $f$ is strictly decreasing in the interval $[0,\bar C]$,
with
\[
\bar C= \frac{A}{P} \frac{C_1-C_2 \frac{\sin \gamma}{\sin
    \theta}}{\frac{\sin \gamma}{\sin \theta}\sin(\gamma+\theta)}
\]
which is strictly positive, being $l_{1,H}>l_{2,H}$. This implies that, for
$l_1 < \bar C $,
\begin{equation}\label{eq:appr2}
 \frac {P^2}{A} > \frac{(P-\delta \tilde P)^2}{A-\delta \tilde A}.
\end{equation}
On the other hand, by Remark~\ref{rem:corda} we get
\[
\delta P \ge \delta \tilde P.
\]
Hence, being obviously $\delta \tilde A\geq\delta A$, by
\eqref{eq:appr2}, it follows \eqref{eq:contra2} for a suitable choice
of $Q$.

\noindent{\bf Case 3.} Finally, if $|E|= |\Omega|/2$, we can both
consider, as minimum sets, $E$ and $\Omega\setminus E$. Hence, if the
angle condition is not verified, we can suppose, without loss of
generality, that the lines $r,s$ and $t$ verify the hypotheses of case
2.

If $\de E \cap \Omega$ is a straight segment, or it is strictly
concave towards $E$, we can repeat line by line the same argument of
case 2.
Otherwise, if $\de E \cap \Omega$ is strictly concave towards
$\Omega\setminus E$, proceeding as in case 1 we construct the straight
segment $QQ_0$, and another straight segment $BC$ joining two suitable
points of $\de E\cap \Omega$. Let $D_1$ and $D_2$ be as in Figure
\ref{fig:fine}, and define $E^*=(E\setminus D_1)\cup D_2$. Choosing
$B,C$ and $Q$ in such a way that $|E|=|E^*|$, since $P_H(E^*;\Omega)<
P_H(E;\Omega)$ we obtain a contradiction, and the proof of the Theorem
is completed when $T$ is a regular point of $\de \Omega$.
\begin{figure}[h]
\includegraphics{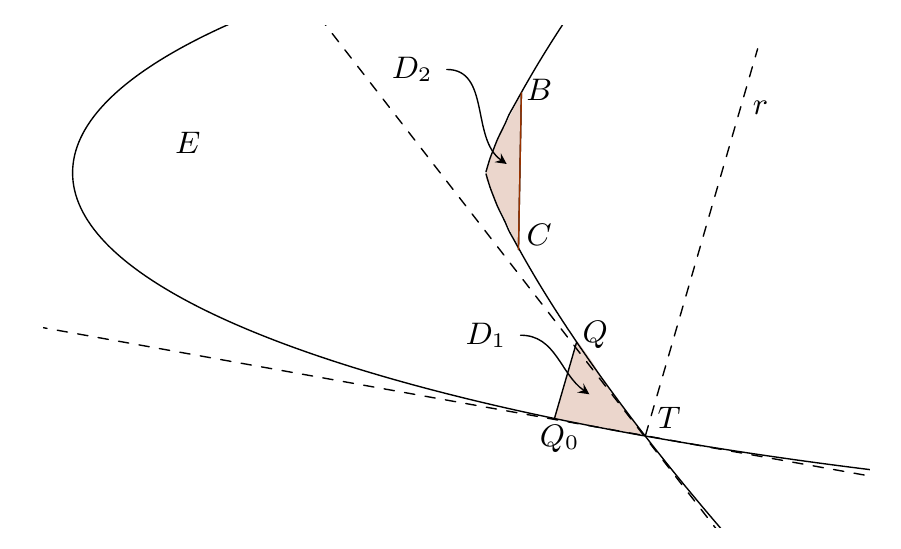}
\caption{}\label{fig:fine}
\end{figure}
Finally, we show that $\de E\cap \Omega$ cannot
join $\de \Omega$ at a non regular point.

By contradiction, suppose that  $\de \Omega$ is not regular at $T$.
By convexity it has different right and left tangent
straight lines, that we denote by $s_1$ and $s_2$ respectively.

Clearly, the tangent line $t$ does not satisfy the contact angle
condition with both $s_1$ and $s_2$. So we can repeat the arguments just
considered by replacing the straight line $s$ with $s_1$ or $s_2$, and
obtaining a contradiction with the minimality of $E$.
\end{proof}

\begin{prop}\label{properties} Let $\Omega$ be an open bounded convex
  set of $\R^2$, $0<k\le|\Omega|/2$,
  and set $E_k$ be a minimizer of problem
\[
 \min \{P_H(F;\Omega),\, F\subset \Omega\text{ and }|F|= k \}.
\]
We have the following properties:
\begin{enumerate}
\item \label{item1} neither $E_k$ nor $\Omega\setminus E_k$ is
  homothetic to a Wulff shape;
\item \label{item2} if $k<|\Omega|/2$, and $T_1$ and $T_2$ are the
  terminal points of $\de E_k\cap \Omega$ on $\de \Omega$,
 then the left and right tangent straight lines at $T_1$ to $\de
 \Omega$ do not make a cone towards $\Omega \setminus E_k$ with the
 analogous lines at $T_2$.
\item \label{item3}  if $k<|\Omega|/2$ and $\de E_k \cap \Omega$ is not
  a straight segment, $\de E_k \cap \Omega$ is concave towards $E$.
\end{enumerate}
\end{prop}
\begin{proof}
We prove the three properties by contradiction with the minimality of
$E_k$, finding a set with same area and smaller perimeter.

Let $E_k$ or $\Omega \setminus E_k$ be homothetic to a Wulff
shape. Since the perimeter $P_H(E_k;\Omega)$ is invariant up to
translations in $\Omega$, we can suppose that  $\de E_k$ touches
at least at one (regular) point $P\in \de \Omega$, and there exists a
small ball $B_P$ centered at $P$ such that $B_P\cap \de E_k\not\subset
\de\Omega$.

We stress that in $P$ the contact angle condition cannot hold. Indeed
$\nu_{E_k}(P)=\nu_\Omega(P)$, and by \eqref{eq:angle} and the
homogeneity of $H$ we should have
that
\[
0 = \langle n_{E_k}(P), \nu_\Omega(P) \rangle = \langle \nabla
H(\nu_\Omega(P)), \nu_\Omega(P) \rangle = H(\nu_\Omega(P)),
\]
so $\nu_\Omega=0$ and this is absurd. Then arguing as in case 3 of the
proof of Theorem \ref{thm:angle}, being $E_k$ (or $\Omega \setminus
E_k$) strictly convex we can add and subtract two small regions in
order to get a new set with the same area and smaller perimeter (see
Figure \ref{fig:fine}). This proves \eqref{item1}.

Property \eqref{item2} easily follows by the convexity of
$\Omega$. Indeed, if $E_k$ has measure smaller than $|\Omega|/2$ and
does not verify \eqref{item2}, we can do a  suitable translation $\de
E_k^t$ of $\de E_k$ towards the vertex $V$ of the cone in $\R^2$, in such
a way that the set $\tilde E$ bounded by $\de E_k^t \cap \Omega$
towards $V$ and $\de\Omega$, has measure $k$ and smaller perimeter than
$E_k$ in $\Omega$ (see figure
\ref{fig:conocontr}). This contradicts the minimality of $E_k$.
\begin{center}
\begin{figure}[h]
  \includegraphics{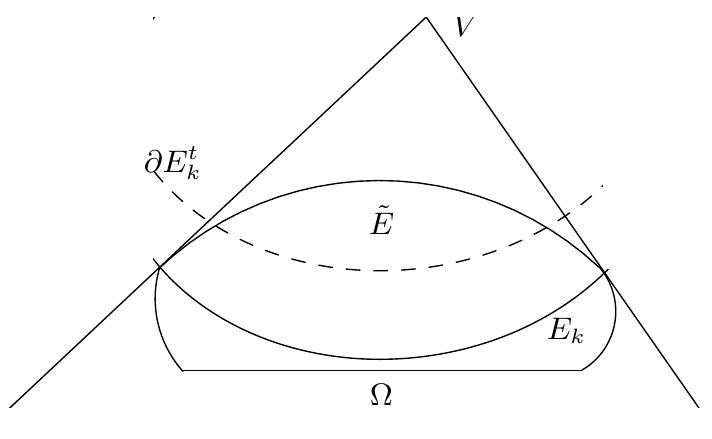}
\caption{ }\label{fig:conocontr}
\end{figure}
\end{center}

Finally, suppose that $\de E_k$ is  concave towards $\Omega\setminus
E_k$. By property \eqref{item2}, the tangent straight lines at
terminal points of $\de E_k\cap \Omega$ either make a cone towards
$E_k$ or are parallel. As in Proposition \ref{prop:conv}, in both
cases we can add 
a small region to $E_k$ in order to decrease the perimeter and,
similarly as in the proof of property \eqref{item2}, with a suitable
translation of $\de E_k\cap \Omega$ towards the vertex of the cone,
keep fixed the area $|E_k|$. This proves property \eqref{item3}.
\end{proof}
In order to prove the existence of a minimizer of \eqref{eq:ratio}, we
need the following technical lemma.
\begin{lemma}\label{lem:decr}
Let $\mu\colon]0,+\infty[\rightarrow\R$ be a lower semicontinuous
function. Suppose that for any $k>0$ there exists $\delta_k>0$ such
that  
\begin{equation}
\mu(k+\delta) \le \mu(k),\quad \text{for any } \delta \in[0,\delta_k].
\end{equation}
Then $\mu$ is decreasing in $]0,+\infty[$.
\end{lemma}
\begin{proof}
By contradiction, suppose that there exist $k_1<k_2$ such that
\begin{equation}\label{eq:3434}
\mu(k_1)<\mu(k_2).
\end{equation}
Define $\varphi(k)$ as
\[
\varphi(k)=
\begin{cases}
\mu(k_1) &\text{if }k\le k_1,\\
\mu(k) &\text{if } k_1 <k < k_2,\\
\mu(k_2) &\text{if } k\ge k_2.
\end{cases}
\]
The function $\varphi$ is lower semicontinuous, and for any
$k$ there exists $\delta_k>0$ such that $\varphi(k+\delta) \le
\varphi(k)$, for any $\delta \in[0,\delta_k]$. 
Hence, we can define $\bar \delta>0$ as
\[
\bar\delta=\sup\{\delta >0\colon \varphi(k_1+\delta)\le
\varphi(k_1)\}.
\]
If $\bar\delta=+\infty$, then $\varphi(k_2)\le
\varphi(k_1)$, and this contradicts \eqref{eq:3434}. Hence,  
suppose that $\bar\delta<+\infty$. Being $\varphi$ lower
semicontinuous, $\bar \delta$ is actually a maximum: 
\[
\varphi(k_1+\bar\delta) \le \liminf_{\delta\rightarrow \bar\delta}
\varphi(k_1+\delta) \le \varphi(k_1).
\]
But this contradicts the definition of $\bar\delta$. Indeed, by the
property of $\varphi$ we can take $\tilde\delta>\bar\delta$
such that $\varphi(k_1+\tilde\delta)\le \varphi(k_1+\bar\delta)\le
\varphi(k_1)$. Hence, necessarily $\mu(k_1)\ge
\mu(k_2)$, and the proof is concluded.
\end{proof}
\begin{theo}\label{thm:decr}
Let $\Omega$ be an open bounded convex set of $\R^2$. Let $\mu(k)$ be
the function defined in  $]0,|\Omega|/2]$ as
\begin{equation}\label{fixA}
  \mu(k)= \min \left\{ \frac{ P^2_H(F;\Omega)}{k},\, F\subset
    \Omega\text{ and }|F|= k \right\}.
\end{equation}
Then, we have the following results hold:
\begin{enumerate}
\item $\mu(k)$ is a decreasing lower semicontinuous
  function in 
  $]0,|\Omega|/2]$, \label{en:1}
\item the sets which minimize \eqref{fixA} verify the contact angle
  condition. More precisely, they verify the thesis of Theorem
  \ref{thm:angle}. \label{en:2}
\end{enumerate}
\end{theo}
\begin{proof}
We first prove that the function $\mu$ is lower
semicontinuous in $]0,|\Omega|/2]$. 

Let be $k\in ]0,|\Omega|/2]$, and take a positive sequence  $k_n$ such
that $k_n\rightarrow k$. Consider $E_n\subset \Omega$ such that
$|E_n|=k_n$ and $\mu(k_n)=Q(E_n)= k_n^{-1} P^2_H(E_n;\Omega)$. By
Proposition~\ref{properties}, $E_n$ is convex. Hence, by the Blaschke
selection Theorem (see \cite{s}, page 50) $E_n$ converges (up to a
subsequence) to a set $E$ in the Hausdorff metric. 
Being $E_n$ convex and bounded, then $\chi_{E_n}\rightarrow \chi_{E}$
in $L^1(\Omega)$ strongly, and $|E|=k$. Using the lower semicontinuity
of $P_H(\,\cdot\,;\Omega)$ (see \cite{ab}) we get
\[
\mu(k)\le Q(E)\le \liminf_n \frac{P_H^2(E_n;\Omega)}{k_n}=\liminf_n
\mu(k_n). 
\]
In order to prove that $\mu$ is decreasing, let be
$k\in]0,|\Omega|/2[$ fixed and consider $E_k$, $|E_k|=k$ such that 
$\mu(k)=Q(E_k)$.

We claim that there exists a positive number $\delta_k$ and a family
of sets $E_k(\delta)$, $0<\delta\le \delta_k$ with continuously
increasing area and $Q(E_k(\delta)) \le Q(E_k)$.
Then
\begin{equation}\label{eq:scendo}
  \mu(|E_k(\delta)|) \le Q(E_k(\delta)) \le \mu(k), \quad \delta
  \in]0,\delta_k].
\end{equation}
Being $\mu$ lower semicontinuous in $]0,|\Omega|/2]$, by Lemma
\ref{lem:decr}  this is sufficient to show that $\mu$ is decreasing. 

By Theorem \ref{th:Wulff-arc}, $\de E_k\cap \Omega$ is a straight
segment or a Wulff arc, and by property \eqref{item1} of Proposition
\ref{properties}, it has two terminal points $T_i$ on $\de \Omega$. We
suppose that such points are regular for $\de \Omega$, so that by
property \eqref{item2} Proposition \ref{properties}, the tangent lines
to $\de \Omega$, $s_i$ at $T_i$ either are parallel or make a cone $A$
towards $E_k$. In the first case, the claim follows immediately by the
convexity of $\Omega$ and making a suitable translation of $\de E_k$.
Hence, we consider the second case, and suppose without loss of
generality that $s_1\cap s_2=(0,0)$. Moreover, by property
\eqref{item3} of Proposition \ref{properties},  $\de E_k\cap \Omega$
is a straight segment, or concave towards $E_k$.

We need to distinguish two cases for the shape of $\Omega$.

\noindent{\bf Case 1.}
$\de E_k \cap \de \Omega$ is not contained in $\de A$.

We set $C(\delta)$, $\delta \ge 0$, the region bounded by
$(1+\delta)\de E_k$ and $\de A$, and $E_k(\delta)=C(\delta)\cap
\Omega$. For sake of simplicity, we define $C(0)=C$.

Let $A_i(\delta)$ be the boundary point of $\de C(\delta)\cap
A$ on $s_i$. Moreover, let be $B_i(\delta)=\de\Omega \cap w_i $,
where $w_i$ is the tangent line to $\de C(\delta)$ at $A_i(\delta)$.
(see figure \ref{fig:angmezzo}).
\begin{figure}[h]
\begin{center}
\includegraphics{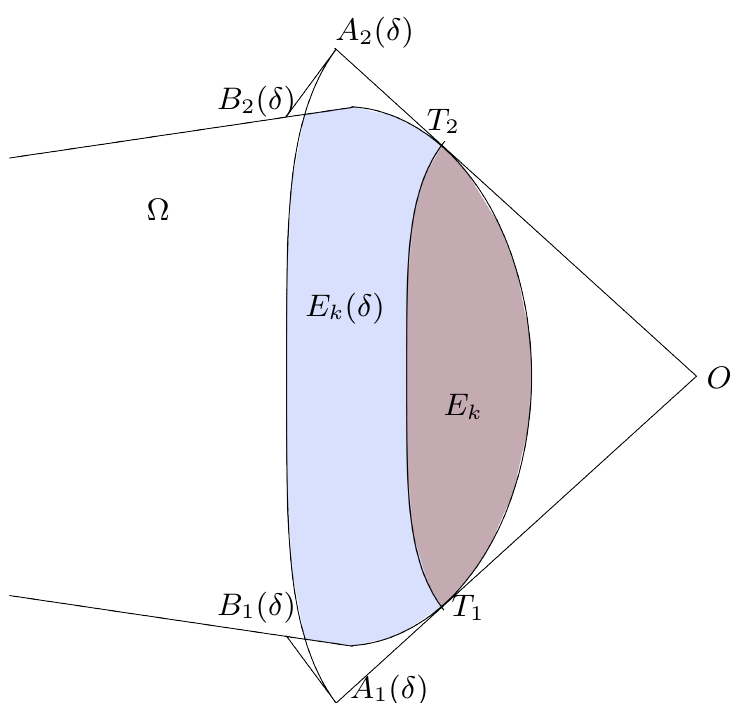}
\end{center}
\caption{}\label{fig:angmezzo}
\end{figure}

Now we compute area and relative perimeter of $E_k(\delta)$.
Observe that the triangles $D_i$ of vertex $A_i(\delta)$,
$B_i(\delta)$ and $T_i$ have area $|D_i|=o(\delta)$.
We have:
\[
|E_k(\delta)| \ge |E_k| + (|C(\delta)|-|C|)  + o(\delta)=
|E_k|+2 \delta |C| +o(\delta)
\]
and
\[
P_H(E_k(\delta);\Omega) \le P_H( C
(\delta);A)=(1+\delta)P_H(C;A)=(1+\delta)P_H(E_k;\Omega).
\]
It follows that
\begin{multline}\label{mult:mezzo}
\frac{1}{\delta}\left[Q(E_k(\delta))- Q(E_k) \right] \le \\ \le
\frac 1 \delta Q(E_k) \left[\frac{(1+\delta)^2}
      {1+2\delta \frac {|C|}{|E_k|}+o(\delta)} - 1 \right] =\\=
  Q(E_k) \left[
\frac{2 \left(1- \frac {|C|}{|E_k|} \right) +o(1) }{1+o(1)} \right].
\end{multline}
Since $|E_k| < |C|$, then for $\delta$ sufficiently small we obtain
that the left-hand side of \eqref{mult:mezzo} is negative. This proves
\eqref{eq:scendo}, and hence \eqref{en:1}, if $T_i$ are regular points
of $\de \Omega$. If, for example, $T_1$ is not a regular point, we
can repeat the arguments just considered by replacing $s_1$ with the
left or right tangent straight line.

Now we prove \eqref{en:2}. In order to fix the ideas, we consider
the regular point $T_1$ and the straight line $r$ which verifies the
contact angle condition with $s_1$. Let $\alpha_{opt}$ be the angle
between $s_1$ and $r$ towards $E_k$, and $\alpha$ the corresponging angle
between $\de E_k \cap \Omega$ and $s_1$ towards $E_k$. Suppose by
contradiction that $\alpha\ne\alpha_{opt}$.

If $\alpha<\alpha_{opt}$, then the construction made in the proof of
case 2 of Theorem \ref{thm:angle} allows to take $E^*$
such that $|E^*|<|E_k|$ and $Q(E^*)<Q(E_k)$, and this contradicts the
monotonicity of $\mu$. If $\alpha>\alpha_{opt}$, and $\de E_k\cap
\Omega$ is a Wulff arc, as in case 3 of
Theorem \ref{thm:angle} we can add and subtract two sets in order to
decrease the perimeter and to preserve the area, contradicting the
minimality of $E_k$. In the case that $\de
E_k\cap \Omega$ is a straight segment, we can add a small region to
$E_k$ in order to decrease the perimeter and with a suitable
translation, keep fixed the area $|E_k|$. 

Finally, $T_1$ cannot be a singular point for $\de \Omega$. Otherwise,
similarly as observed at the end of the proof of Theorem
\ref{thm:angle} and proceeding as above, we get a contradiction with
the minimality of the minimizer.

\noindent{\bf Case 2.} $\de E_k \cap \de \Omega$ is contained in
$\de A$, that is $E_k=C$. 

Define $E_k(\lambda)=\lambda E_k$, $\lambda \ge 0$, and $r\ge 0$ such
that 
\[ 
\lambda_{max}=\max\{\lambda\ge 0\colon E_k(\lambda) \cap \de \Omega
\subset \de A\}.
\]
First, we prove that at the terminal points of $\de E_k\cap \Omega$ it
holds the contact angle condition \eqref{eq:angle}.

In order to fix the ideas, we consider the regular point $T_1\in \de
\Omega$ and the straight line $r$ which verifies the contact angle
condition with $s_1$. Let $\alpha_{opt}$ be the angle between $s_1$
and $r$ towards $E_k$, and $\alpha$ the corresponging angle between
$\de E_k \cap \Omega$ and $s_1$ towards $E_k$. Suppose by
contradiction that $\alpha\ne\alpha_{opt}$.

\noindent{\bf Case 2-a} Let be $\lambda_{max}>1$. Reasoning as in the
proof of Theorem \ref{thm:angle}, we find $E^*$ such that
$Q(E^*)<Q(E_k)$, with $|E_k|-|E^*|$
sufficiently small. Then there exists $\rho>0$ such that $|\rho
E^*|=\rho^2|E^*|=|E_k|$, and $Q(\rho E^*)=Q(E^*)<Q(E_k)$. This
contradicts the minimality of $E_k$.

Repeating the same argument for $T_2$, we have that the terminal
points of $\de E_k \cap \Omega$ have to verify the angle condition,
that is $E_k$ is homothetic to a Wulff sector $W\cap A$.

\noindent{\bf Case 2-b} Let be $\lambda_{max}=1$. Then, as
$0<\lambda<\lambda_{max}$, the set $\lambda E_k$ is such that
$Q(\lambda E_k)=Q(E_k)$. Thanks to case 2-a, 
we have that $\mu(|\lambda E_k|)$ is 
attained at a Wulff sector,  namely the set $(\tilde \lambda W)\cap
A=(\tilde \lambda W)\cap \Omega$, for $\tilde\lambda>0$ such that
$|\lambda E_k|=|(\tilde \lambda W)\cap A|$. Hence
$\mu(|E_k|)=Q((\tilde\lambda W)\cap \Omega)<Q(E_k)$. Define
\begin{equation}\label{eq:gammamax}
\gamma_{max}=\max\{\gamma\ge 0\colon (\gamma W) \cap \Omega\text{ is
  homothetic to a Wulff sector}\}. 
\end{equation}

We have that $\gamma_{max}$ is finite and $|\gamma_{max} W\cap
\Omega|<|E_k|$. Otherwise, there exists $\gamma\le \gamma_{max}$ such that
$|\gamma W\cap \Omega|=|E_k|$ and $Q(\gamma
W\cap\Omega)=Q(\tilde\lambda W\cap\Omega)<Q(E_k)$, and 
this is a contradiction.

As matter of fact, the homogeneity of $H$ and \eqref{eq:HH0} imply,
for $\xi \in \de W$, that $H(\nu_W(\xi))=\langle
\nu_W(\xi),\xi\rangle$. Moreover,
for $\xi\in \de A$, $\langle 
\nu_A(\xi),\xi\rangle=0$. Hence by the divergence Theorem we get that,
for $\gamma> 0$,
\begin{equation}\label{npform}
P_H(\gamma W; A) = 2\gamma|W\cap A|.
\end{equation}

Define $E(\delta)=\Omega\cap[(\gamma_{max}+\delta)W]$, and $A_\delta$
the cone made by the two half-straight lines $s_i^\delta$, $i=1,2$
with origin at $(0,0)$ and passing through one of the two
terminal points of $\de[E(\delta)\cap\Omega]$.

By \eqref{npform} and the convexity of $\Omega$, we get, for an
appropriate $\delta$,
$|E(\delta)|=k$ and
\[
Q(E(\delta)) \le 4 |W\cap A_\delta| <
4 |W\cap A|= Q(\gamma_{max}W)<Q(E_k).
\]
Then $\de E_k$ must verify the contact angle condition at each
$T_i$, and this concludes the case 2-b, and \eqref{en:2} is proved. 

In order to prove \eqref{eq:scendo}, and hence \eqref{en:1}, we observe
that from \eqref{en:2}, $E_k=(\lambda W) \cap \Omega$, for
some $\lambda>0$. Let $\gamma_{max}$ as in
\eqref{eq:gammamax}, and suppose that $\gamma_{max}=\lambda$, otherwise
\eqref{eq:scendo} is immediate, being $Q(E_k)=Q(\gamma W\cap \Omega)$,
for any $0<\gamma < \gamma_{max}$. Defining
$E(\delta)=\Omega\cap[(\gamma_{max}+\delta)W]$ and reasoning as in case
2-b, we get \eqref{eq:scendo}.

Finally, the regularity of $T_i$ on $\de \Omega$ follows exactly as in
the case 1, and the proof is completed.
\end{proof}
\begin{rem}
We observe that if $E$ is a minimizer of \eqref{eq:ratio}, and 
$|E|<|\Omega|/2$, then $E$ is homothetic to a Wulff sector with sides
on $\de \Omega$. Otherwise, arguing as in case 1 of the proof of
Theorem \ref{thm:decr}, we construct a new set $E^*$ with
$Q(E^*)<Q(E)$. 
Hence, $E=E(\lambda)=A\cap(\lambda W)$ with sides on $\de \Omega$. Being
\[
Q(E(\rho))= 4 |W\cap A|,\quad \forall\, \rho\colon |E(\rho)|\le
\frac{|\Omega|}{2},
\]
where $E(\rho)=A\cap (\rho W)$, there exists another minimizer
$F$ which is a Wulff sector with sides on $\de \Omega$ and 
$|F|=|\Omega|/2$.
\end{rem}
Now we are able to prove the main result.
\begin{theo}\label{thm:main}
  Let $\Omega$ be an open bounded convex set of $\mathbb R^2$. Then
  there exists a convex minimizer of problem \eqref{eq:ratio} whose
  measure is equal to $|\Omega|/2$. More precisely, either a minimizer
  $E$ of \eqref{eq:ratio} has measure $|\Omega|/2$, or $E$ is
  homothetic to a Wulff sector with sides on $\de\Omega$. Finally, it
  verifies the contact angle condition.
\end{theo}
\begin{proof}
Let $\mu$ defined as in the above theorem and, being $\mu$ decreasing in
$]0,|\Omega|/2]$, it attains its minimum at $k=|\Omega|/2$. 

Now we are able to prove that \eqref{eq:ratio} has a minimum. Let
$\tilde E$ be such that $|\tilde E|=|\Omega|/2$ and 
$\mu(|\Omega|/2)=Q(\tilde E)$. Let $E_n$, $n\in\mathbb N$ be a
minimizing sequence of problem \eqref{eq:ratio}, that is
\[
\lim_n Q(E_n)= C_H,\quad 0<|E_n|\le
|\Omega|/2.
\]
Without loss of generality, we may suppose that, for
any $n\in \mathbb N$, $Q(E_n)=\mu(|E_n|)$. Otherwise, we replace $E_n$
with the minimizer of problem \eqref{eq:pb} with volume constraint
$k=|E_n|$. Then
\[
C_H\le Q(\tilde E)=\mu(|\Omega|/2) \le \mu(|E_n|)=Q(E_n).
\] 
Passing to the limit,
\[
C_H=\mu\left(\frac{|\Omega|}{2}\right),
\]
and $\tilde E$ is a minimizer of \eqref{eq:ratio}, whose boundary in
$\Omega$ is a straight segment or a Wulff arc. From the proof of
Theorem \ref{thm:decr} it follows that if $E$ is another minimizer of
\eqref{eq:ratio} with $|E|<|\Omega|/2$, then it is a Wulff sector with
sides on $\de\Omega$. 
Recalling Theorem \ref{thm:angle}, the result is completely proved. 
\end{proof}
In the following theorem, we characterize the minimizers for
centrosymmetric sets, and find the constant $C_H$ in \eqref{eq:ratio}.

For sake of simplicity, if $T$ is a point in $\R^2$, we put
$L_H(T)=\mathcal L_H(\gamma)$, where $\mathcal L_H$ is defined in
\eqref{length}, and $\gamma$ is a curve which represent the straight
segment $OT$ joining $T$ with the origin $O$. We observe that if
$T=(x,y)$, then $L_H(T)=H(-y,x)$.

\begin{theo}\label{thm:cost}
Let $\Omega\subset\mathbb R^2$ be a convex bounded set, symmetric
about the origin $O$. Then a minimizer of \eqref{eq:ratio} is
a set $E$ whose boundary in $\Omega$ is a straight segment passing
through the origin and such that $P_H(E;\Omega)=2 r_H$, where
$r_H=r_H(\Omega)=\min_{T\in \de \Omega} L_H(T)$. Hence, 
\[
C_H = \frac{8r_H^2}{|\Omega|}.
\]
\end{theo}
\begin{proof}
The first step is to prove the existence of a set $E$ enjoying the
properties of the statement. Let us consider the set
\[
B(r_H)=\{(x,y)\in\R^2\colon L_H(x,y)<r_H\}.
\]
Then $\de B(r_H)$ meets $\de \Omega$ at least at two symmetric regular
points $T_1$, $T_2$. We observe that in $T_i$ the contact angle condition
is satisfied. Indeed, the anisotropic outer normal to the straight
segment $OT_i$ is $n_E(T_i) =\nabla H (-y_i,x_i)$, where $T_i=(x_i,y_i)$,
$i=1,2$. 
Denoted by $\nu_\Omega(T_i)$ the unit outer normal to $\de \Omega$ at
$T_i$, being $\nu_\Omega(T_i)=(H_y(-y_i,x_i),-H_x(-y_i,x_i))$,
we have $\langle n_E(T_i), \nu_\Omega(T_i)\rangle=0$.

We show that $T_1 T_2$ is the boundary in $\Omega$ of the
required set $E$, and $Q(E)=\frac{8r_H^2}{|\Omega|}$. 

By Theorem \ref{thm:main}, there exists a convex minimizer of
\eqref{eq:ratio} whose measure is $|\Omega|/2$, which is a straight
segment or a Wulff arc. If we show that $P_H(E;\Omega)\le
P_H(F;\Omega)$, where $F$ is a open convex subset of $\Omega$ such
that $|F|=|\Omega|/2$ and $\de F\cap \Omega$ is a straight segment or
a Wulff arc, we have done.

Clearly, any straight segment passing through the origin bounds in
$\Omega$ a set with greater perimeter than $E$ and with same area
$|\Omega|/2$. We do not consider the straight segments which not
contain the origin, because they bounds in $\Omega$ sets with measure
different from $|\Omega|/2$. Hence we can suppose that $\de F\cap
\Omega$ is a Wulff arc.

Obviously, $O\not \in \de F$, otherwise $|F|\ne |\Omega|/2$. More
precisely, denoted by $P_1$ and $P_2$ the terminal points of $\de
F\cap \Omega$, we get that $O\in F\setminus \bar G$, where
$G\subset F$ is bounded by $\de \Omega$ and $P_1P_2$, otherwise
$|\Omega|/2| \le |G|<|F|$, and this is impossible.
Hence we can consider the straight segments in $F$,
${OP_1}$ and ${OP_2}$, and it is not difficult to show that
\[
P_H(F;\Omega) > L_H (P_1)+ L_H(P_2)\ge 2r_H = P_H(E;\Omega),
\]
and this concludes the proof.
\end{proof}
\begin{rem}\label{rem:thm}
  If $\Omega=\{(x,y)\colon H(-y,x)< r\}$, i.e. $\Omega$ is obtained
  by a rotation of $\frac \pi 2$ the $r-$level set of $H$, then
  Theorem \ref{thm:cost} gives
  \[
  C_H = \frac{8r^2}{|\Omega|}=\frac{8}{\kappa_H},
  \]
  where $\kappa_H=|\{ (x,y)\colon H(x,y)<1\}|$. Observe that any
  straight segment passing through the origin and joining the boundary
  of $\Omega$ bounds a minimizer.
 
In particular, if $H(x,y)=H^o(x,y)=(x^2+y^2)^{1/2}$, we recover the
classical result $C_H=\frac{8}{\pi}$ (see for instance
\cite{maz},\cite{cia}).
\end{rem}

\section{Some examples}
Here we apply the results just obtained to some particular function
$H$.
\begin{ex}
Let $H(x,y)$ defined as
\[
H(x,y)= \left(
\frac{x^2}{a^2}+\frac{y^2}{b^2}\right)^\frac{1}{2}.
\]
An immediate calculation gives that
\[
H^o(x,y)= \left(a^2x^2+b^2y^2\right)^\frac{1}{2}.
\]
If $\Omega$ is the ellipse $\Omega\!=\!\{(x,y)\colon\!
H^o(x,y)\!<r\}$, then $\Omega\!=\!\{(x,y)\colon\!
H(-y,x)\!<\!\frac{r}{ab}\}$, and $|\Omega|=\frac{\pi r^2}{ab}$.
By Theorem \ref{thm:cost} and Remark \ref{rem:thm} we have
\begin{equation}
  \label{eq:isoel} P_H^2(E;\Omega) \ge
  \frac{8}{\pi ab} |E|, \qquad
  \forall E \subset \Omega \colon |E| \le \frac{\pi r^2}{2ab}.
\end{equation}
Moreover, the equality in \eqref{eq:isoel} holds if and only if $\de E
\cap \Omega$ is any straight segment passing through the origin (see
Figure \ref{fig:ellisse}).

We observe that if we compute $C_H$ for the ellipse
$\Omega_1=\{(x,y)\colon H(x,y)<r\}$, with for example, $a>b$, then the
smaller axis of the ellipse (in the usual sense) is the boundary of
the only minimizer of \eqref{eq:ratio} (see
Figure \ref{fig:ellisse}), and the constant $C_H$ is
\[
C_H = \frac{8}{\pi ab}\frac{b^2}{a^2}.
\]

\begin{figure}[h]
 \includegraphics[scale=.3]{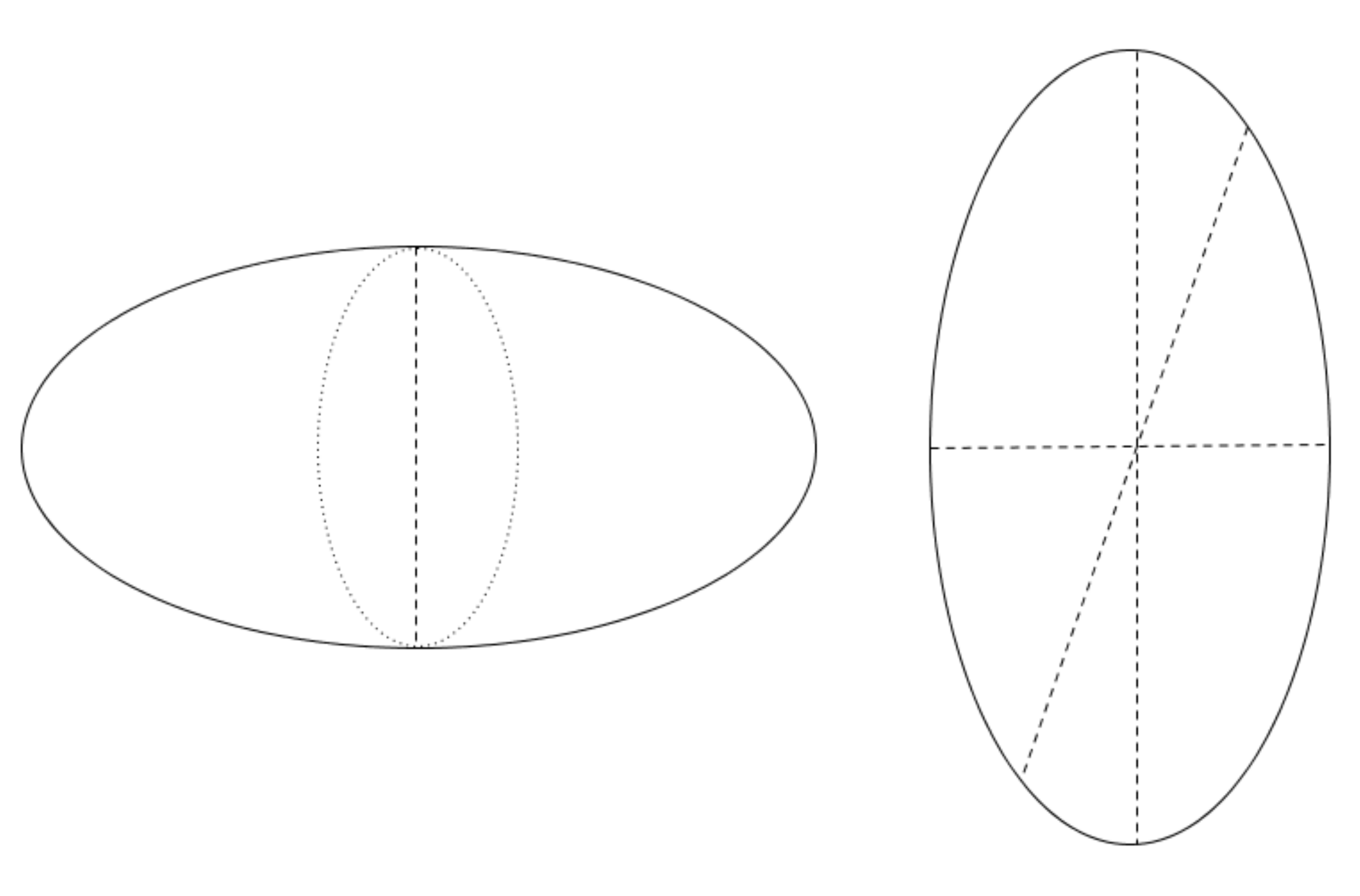}
\put(-7.8cm,1.37cm){$\Omega_1$}
\put(-.35cm,1cm){$\Omega$}
\caption{In the first figure, $\Omega_1$ is a level set of
  $H$, and the straight segment is the boundary of the only
  minimizer of~\eqref{eq:ratio}. In the second figure, $\Omega$ is a
  level set of $H^o$, and any straight segment passing through the
  origin is the boundary of a minimizer.} \label{fig:ellisse}
\end{figure}

We point out that the above result for $\Omega$ can be obtained
directly by the classical relative isoperimetric inequality for the
Euclidean perimeter. Indeed, the anisotropic relative perimeter of a
smooth set $E$, whose boundary is described by $(u(t),v(t))$, with $t \in
[\alpha,\beta]$, is
\begin{equation}
\label{eq:ell2}
P_H(E;\Omega)=\int_\alpha^\beta H(-v', u' )
\,dt=\int_\alpha^\beta\left(
\frac{(v')^2}{a^2}+\frac{(u')^2}{b^2}\right)^{\frac{1}{2}}\,dt.
\end{equation}
Defining $w=au$ and $z=bv$, the curve $(w(t),z(t))$ describe the
boundary of the unit Euclidean disk $B_r$ with radius $r$ and
centered at the origin. By changing the variables in \eqref{eq:ell2},
we get
\[
\int_\alpha^\beta
\left(\frac{(z')^2}{a^2b^2}+\frac{(w')^2}{a^2b^2}\right)^{\frac{1}{2}} dt
= \frac{1}{ab}P(\tilde E, B_1) \ge \frac{1}{ab}\sqrt{\frac{8}{\pi}}
|\tilde E|^\frac{1}{2}= \sqrt{\frac{8}{\pi ab}}|E|^{\frac 1 2},
\]
where $\tilde E$ is the set obtained by $E$ after the change of
variables. Being $|\tilde E|= ab |E|$, we get \eqref{eq:isoel}.

Finally, the characterization of the minimizers is a direct
consequence of the fact that in the classical relative isoperimetric
inequality, the minimizers are the diameters. Hence in this case we
get the relative anisotropic isoperimetric inequality by a linear
trasformation, as a consequence of the classical relative
isoperimetric inequality.
\end{ex}
\begin{ex}
Now suppose that
\[
H(x,y)=(|x|^p+|y|^p)^{\frac{1}{p}}.
\]
where $2\le p < +\infty $ and $p'=\frac{p}{p-1}$.
Hence, we have $H^o(x,y)=(|x|^{p'}+|y|^{p'})^{\frac{1}{p'}}$.

Let us consider $\Omega=\{(x,y)\colon |x|^p+|y|^p<r^p\}$. Being
$\Omega$ invariant by $\frac \pi 2-$rotations, by Theorem
\ref{thm:cost} and Remark \ref{rem:thm} we have
\[
  \label{eq:isoeli} P_H^2(E;\Omega) \ge
  {\frac{8}{\kappa_H}} |E|, \qquad \forall E\subset\Omega\colon |E|
  \le \frac{r^2 \kappa_H}{2},
\]
where $\kappa_H=|\{(x,y)\colon H(x,y)<1\}|$, and any straight segment
passing through the origin bounds a minimizer.
\end{ex}
\begin{ex}\label{ex:3}
Let $H$ be defined as follows:
\[
H(x,y)=
\begin{cases}
 (|x|^p+|y|^p)^{1/p} & \text{if } xy\ge 0, \\
(|x|^q+|y|^q)^{1/q} & \text{if } xy\le 0,
\end{cases}
\]
with $p>2$, $q>2$ and $p>q$. Let us consider $\Omega=\{(x,y)\colon
H(-y,x)<r\}$. Then
\[
C_H=C_H(\Omega)=\frac{8}{\kappa_H}.
\]
We stress that if $\Omega_1=r\{(x,y)\colon H(x,y)<r\}$, then easy
computations give that
\[
C_H=C_H(\Omega_1)= \frac{8}{\kappa_H} 4^{\frac 1 p - \frac 1 q}.
\]
Observe that $C_H(\Omega)>C_H(\Omega_1)$ (compare Figure
\ref{fig:centrosimmdoppio}).
\begin{figure}[h]
\begin{center}
  \includegraphics[scale=.2]{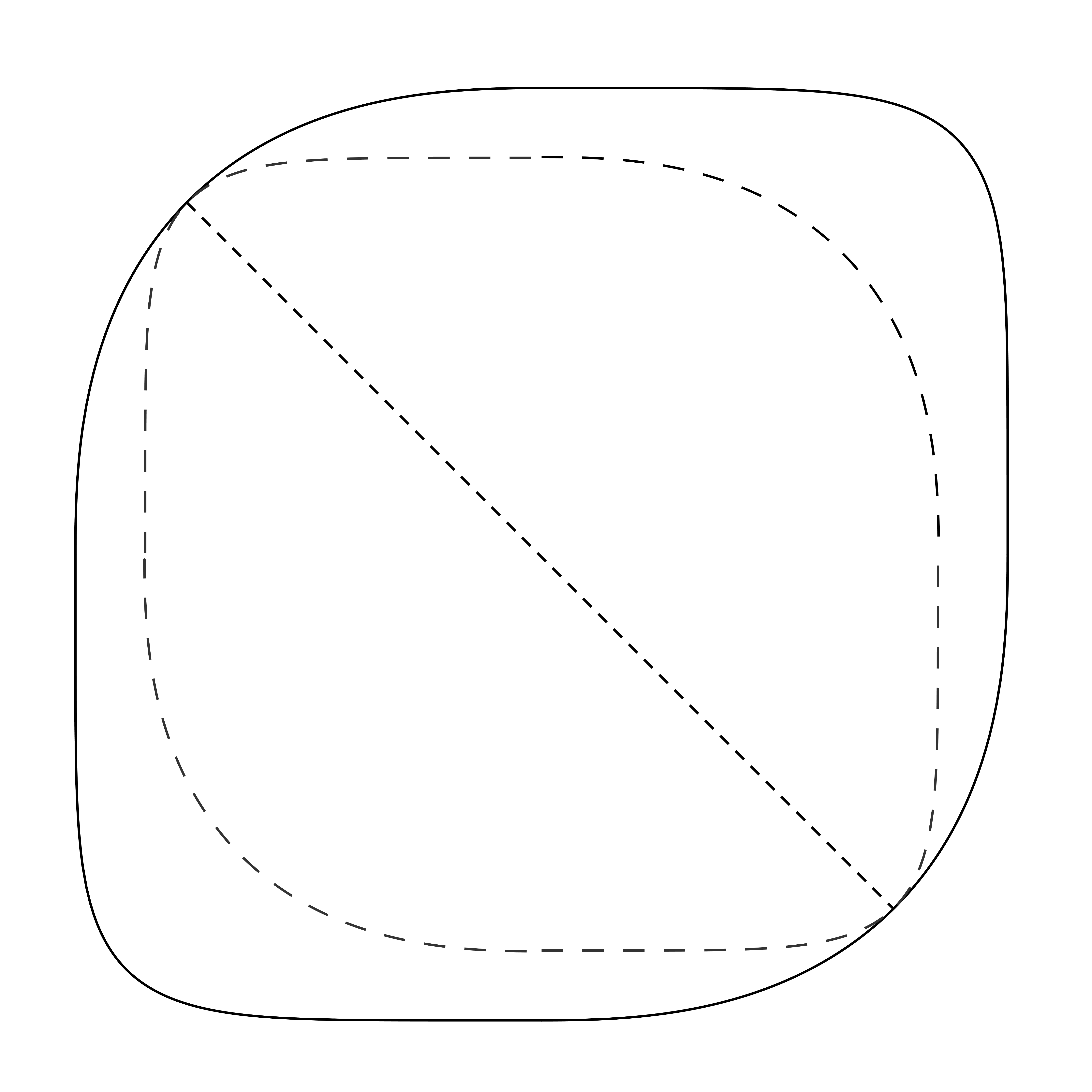}
\end{center}
\caption{Example \ref{ex:3}.
The solid line represents a level set of
  $H$, while the straight segment is the boundary of the only
  minimizer of~\eqref{eq:ratio}.} \label{fig:centrosimmdoppio}
\end{figure}
\end{ex}
\begin{ex}[A non-regular case]\label{nonreg}
Let us consider $H(x,y)=\max\{|x|,|y|\}$. The singular behavior of $H$
does not allow to apply the previous results. Then, in order to prove the
anisotropic isoperimetric inequality relative to $\Omega$ with respect
to $H$, we argue by approximation.

Let be $\Omega=\{(x,y)\colon \max\{|x|,|y|\}<r\}$, and
$H_p(x,y)=(|x|^p+|y|^p)^{1/p}$. For any set $E\subset \Omega$ such
that $|E|\le 2r^2$, we have
\begin{equation}\label{eq:inftyp}
  P^2_{H_p}(E;\Omega)\ge 2 |E|,
\end{equation}
and the best constant is reached by a rectangle whose boundary in
$\Omega$ is the straight segment joining $(-r,0)$ and $(r,0)$ (or
$(0,-r)$ and $(0,r)$). We can pass to the limit as $p\rightarrow
+\infty$ in \eqref{eq:inftyp}, obtaining
\begin{equation}\label{eq:quadrato}
P_H^2(E;\Omega)\ge 2 |E|,\qquad \forall E\subset\Omega\colon |E|\le
2r^2.
\end{equation}
\begin{figure}[h]
\begin{center}
  \includegraphics[scale=.75]
{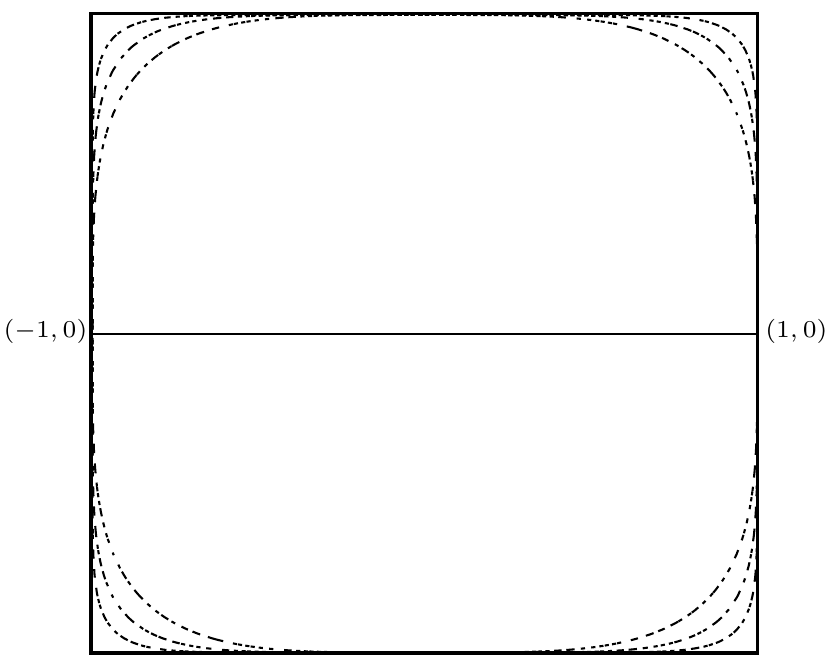}
\end{center}
\caption{Example \ref{nonreg}}
\end{figure}
Any straight segment passing through the origin and joining the
boundary of $\Omega$ bounds a minimizer. Unlike the case of $H$
smooth (compare Remark \ref{rem:thm}), such sets are not the only
minimizers. 

For example, in Figure \ref{fig:minim} some minimizer is
represented. Indeed, if $\de E$ is described by a Lipschitz function
$u(t)$, $t\in[a,b]$, the perimeter is  
\[
P_H(E;\Omega)=\int_a^b \max\{1,|-u'(t)|\}dt.
\]
Then in the picture on the left-hand side of Figure \ref{fig:minim}, the
perimeter of $E$ is $2r$ and $|E|=2r^2$. Moreover, in the other
picture any triangle $E$ such that $\de E \cap \Omega$ is a straight
segment parallel to a diagonal is a minimizer.  
\begin{figure}[h]
\begin{center}
\includegraphics{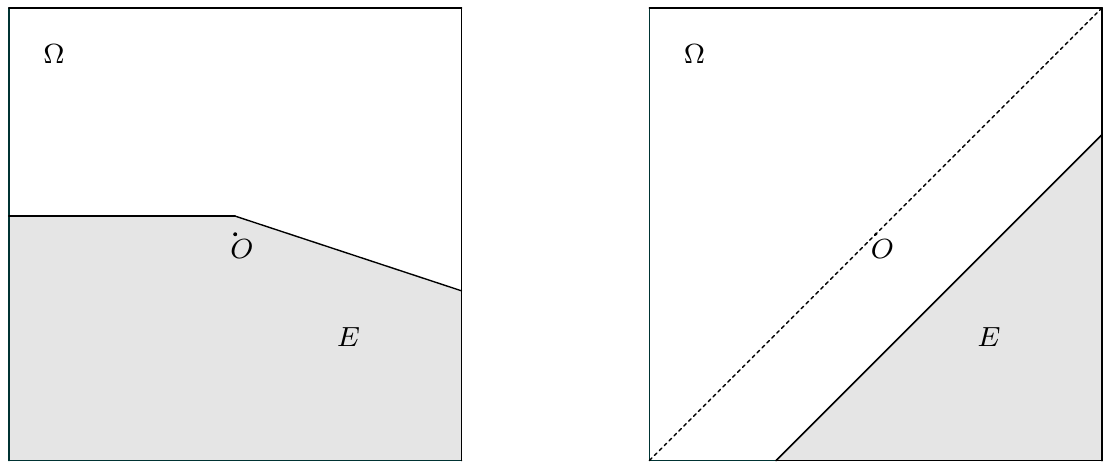}
\end{center}
\caption{}
\label{fig:minim}
\end{figure}
\end{ex}

\end{document}